\documentclass{birkjour}

\usepackage{amsmath,amssymb,amscd,amsthm}
\usepackage[all]{xy}
\usepackage{tikz}
\usetikzlibrary{matrix,arrows,decorations.pathmorphing}

\newcommand{\R}{\ensuremath{\mathbb{R}}}

\newcommand{\Z}{\ensuremath{\mathbb{Z}}}
\newcommand{\Na}{\ensuremath{\mathbb{N}}}
\newcommand{\K}{\ensuremath{\mathbb{K}}}

\newcommand{\T}{\ensuremath{\mathbb{T}}}

\newtheorem{teo}{Theorem}[section] 
\newtheorem{conjectura}{Conjecture}[section] 

\newtheorem{defi}[teo]{Definition} 
\newtheorem{obs}[teo]{Remark} 
\newtheorem{prop}[teo]{Proposition}

\DeclareMathOperator{\sen}{sen}

\DeclareMathOperator{\coin}{Coin} 
\DeclareMathOperator{\bucoin}{BUCoin}

\DeclareMathOperator{\NBU}{NBU}

\DeclareMathOperator{\ind}{ind}

\begin{document}

\title{Nielsen-Borsuk-Ulam number for maps between tori}

\author[Givanildo]{Givanildo Donizeti de Melo}


\email{givadonimelo@hotmail.com}

\thanks{The first author was supported by CAPES - Brazil, the second author was partially supported by FAPESP, Projeto Tem\'atico: Topologia Alg\'ebrica, Geom\'etrica e Diferencial, 2016/24707-4}

\author{Daniel Vendr\'uscolo}
\address{Universidade Federal de S\~ao Carlos\\
S\~ao Carlos - SP \\
Brazil}
\email{daniel@dm.ufscar.br}


\date{January 1, 2004}


\begin{abstract}
We compute the Nielsen-Borsuk-Ulam number for any selfmap of $n-$torus, $\T^n$, as well as any free involution $\tau$ in $\T^n$, with $n \leqslant 3$. Finally, we conclude that the tori, $\T^1$, $\T^2$ and $\T^3$, are Wecken spaces in Nielsen-Borsuk-Ulam theory.  Such a number is a lower bound for the minimal number of pair of points such that $f(x)=f(\tau(x))$ in a given homotopy class of maps.
\end{abstract}

\maketitle

\section{Introduction}

In the literature one can find many different generalizations of the classical Borsuk-Ulam Theorem for maps from the sphere $S^n$ in the Euclidean space $\R^n$. For a good historical review, along with several other aspects, we refer to \cite{UsingBorsukUlam}.

One possible generalization can be the following: given two topological spaces $X$ and $Y$ and a free involution $\tau$ on $X$ we can ask if the triple $(X,\tau;Y)$ has the Borsuk-Ulam Property, i. e., if for any continuous map $f:X\to Y$ there exists a point $x\in X$ such that $f(x)=f(\tau(x))$.

In this context, the classical Theorem just states that $(S^n, antipodal\\ map, \R^n)$ has the Borsuk-Ulam Property. In \cite{dacibergsurfaces} this approach was used to study Borsuk-Ulam Property for surfaces with maps on $\R^2$, and it indicated that the answer may depend on the involution, i. e., the same surface can have this property in respect to an involution $\tau_1$, but not for another involution $\tau_2$.

Accordingly, \cite{daciberg} discusses the triples $(X,\tau,S)$ with Borsuk-Ulam Property where $X$ is a $CW$-complex, $\tau$ is a cellular involution and $S$ is a connected closed surface, using braids group of surfaces.

More recently (in \cite{vinicius}) the Borsuk-Ulam Property was stated not for a triple $(X,\tau,Y)$ but for each homotopy class of selfmaps of surfaces with Euler characteristic zero. It must be noted that for maps on $\R^n$ there is only one such class, however, this is not generally the case.

From this perspective, while investigating for which homotopy class of maps $f:X\to Y$ it is true that for any $f'$ in such class, there exists a point $x\in X$ such that $f'(x)=f'(\tau(x))$, the studies \cite{daniel, daniel2} have taken on a different approach. Using ideas from Nielsen fixed point theory, Nilsen Borsuk-Ulam classes and a Nilsen Borsuk-Ulam number were defined, for a homotopy class of maps between  triangulated, orientable, closed manifolds. Such invariant is a lower bound, in the homotopy class, for the number of pairs of points satisfying  $f(x)=f(\tau(x))$. In the present study we compute the Nielsen-Borsuk-Ulam number for selfmaps of tori until dimension $3$. 

We will define the $n-$torus, denoted by $\T^n$, as the quotient space ${\R^n}/{\Z^n}$ and we remember that $\pi_1(\T^n)=\Z^n$, and that $n-$torus is a  $K(\pi,1)$ space. So  the set of homotopy classes of selfmaps of the $n-$torus, $[\T^n,\T^n]$, is in bijection with the set of homomorphisms, $\hom(\Z^n,\Z^n)$, and this set can be regarded as the $n\times n$-integer matrices. Occasionally, there will be a slight notation abuse between an equivalence class and some of their representatives. We will always see tori as abelian groups, therefore, we will use additive notation for the group operation.

Interestingly, the results presented here show that, for $n\leqslant 3$, the triples $(\T^n, \tau,\T^n)$ do not have the Borsuk-Ulam Property, for any involution $\tau$, and also, that tori are Wecken spaces in Nielsen-Borsuk-Ulam theory. In each case, we present maps that realize Nielsen-Borsuk-Ulam number. 

In the proofs of Theorems \ref{NBUtoroT1}, \ref{NBUtoroT2} and \ref{NBUtoroT3} a similar reasoning was adopted, given a map $f$ that represents a homotopy class, we show that there is a small perturbation of $f$ in this same class, usually called $f'$, such that $f'$ realizes the Nielsen-Borsuk-Ulam number. 

This study consists of 5 sections, this introduction included. Section $2$ gives a brief overview of Nielsen-Borsuk-Ulam theory (following \cite{daniel, daniel2}), Sections $3$, $4$ and $5$ deal with the computation of the Nielsen-Borsuk-Ulam number for selfmaps of tori in dimension $1$, $2$ and $3$, respectively, and Section $6$ presents some considerations for tori of greater dimension.

\section{Nielsen-Borsuk-Ulam theory}

For practical reasons, we will reproduce in this section some definitions and propositions from \cite{daniel} and \cite{daniel2}.

Denoting by $Coin(f,f\circ\tau)$ the coincidence set of the pair $(f,f\circ\tau)$, \cite[Theorem~2.1]{daniel} shows, in the context of simplicial complexes, that we can suppose $Coin(f,f\circ\tau)$ finite. Moreover \cite[Theorem~3.5]{daniel} shows that if two homotopic maps, $f$ and $g$, are such that $Coin(f,f\circ\tau)$ and $Coin(g,g\circ\tau)$ are both finite, then there exists a homotopy between them with such set finite in each level.

\begin{defi}\cite[2.1]{daniel2} Let $(X,\tau;Y)$ be a triple where $X$ and $Y$ are finite $n$-dimensional complexes, $\tau$ is a free simplicial involution on $X$ for any map $ f:X\to Y$ with $Coin(f, f\circ\tau)= \{x_1, \tau(x_1), \ldots, x_m, \tau(x_m)\}$, we define the Borsuk-Ulam coincidence set for the pair $(f,\tau)$, as the set of pairs:
$$BUCoin(f;\tau)=\{(x_1,\tau(x_1));\ldots;(x_m,\tau(x_m))\}$$
and we say that two pairs $(x_i,\tau(x_i)), (x_j,\tau(x_j))$ are in the same Borsuk-Ulam coincidence class if there exists a path $\gamma$ from a point in $\{x_i,\tau(x_i)\}$ to a point in $\{x_j,\tau(x_j)\}$ such that $f\circ \gamma$ is homotopic to $f\circ\tau\circ\gamma$ with fixed endpoints.
\end{defi}

\begin{defi}\cite[2.2]{daniel2} A Borsuk-Ulam coincidence class $C$ is called {\it single} if for one (or any) pair $(x,\tau(x))\in C$ there exists a path $\gamma$ from $x$ to $\tau(x)$ such that $f\circ \gamma$ is homotopic to $f\circ\tau\circ\gamma$ with fixed endpoints.
\end{defi}

\begin{prop}\cite[2.4]{daniel2} A Borsuk-Ulam coincidence class $C$ is single if, and only if, it is composed of just one usual coincidence class of the pair $(f,f\circ\tau)$. Moreover, if $C$ is a finite Borsuk-Ulam coincidence class of the pair $(f,\tau)$ which is not single (called {\it double}), then we can change the labels of the elements of $C$ in a way that:
\begin{itemize}
\item $C=\{(x_1,\tau(x_1)),\dots,(x_k,\tau(x_k))\}$;
\item $C=C_1\cup C_2$ where $C_1$ and $C_2$ are usual coincidence classes of the pair $(f,f\circ\tau)$;
\item $C_1=\{x_1,\dots,x_k\}$ and $C_2=\{\tau(x_1),\dots,\tau(x_k)\}$. 
\end{itemize}
\end{prop}

At the begining of page 615 of \cite{daniel2} there is an observation about the relation between the local indices $ind(f,f\circ\tau;c)$ and $ind(f,f\circ\tau;\tau(c))$ in orientable manifolds. However, in this observation the effect of dimension was neglected. In fact, we have:
$$
ind(f,f\circ\tau;c)= \left\{
\begin{array}{rcl}
(-1)^n ind(f,f\circ\tau;\tau(c))& \mbox{if} & \tau\mbox{ preserves orientation,}\\
(-1)^{n-1} ind(f,f\circ\tau;\tau(c))& \mbox{if} & \tau\mbox{ reverses orientation.}
\end{array}
\right.
$$
where $ind(f,f\circ\tau;c)$ is the usual local index for coincidence and $n$ is the dimension of the manifold.

This change does not affect the results in \cite{daniel2}, however, it will be necessary to reorganize the definition of {\it pseudo-index} used there.

\begin{defi}\cite[Compare with Definition~2.5]{daniel2}
Let $X$ and $Y$ be closed orientable triangulable $n$-manifolds, $\tau$ a free involution on $X$ and $f:X\to Y$ a continuous map such that $BUCoin(f,\tau)$ is finite. If $C=\{(x_1,\tau(x_1)),\dots,\\ (x_k,\tau(x_k))\}$ is a Borsuk-Ulam coincidence class of the pair $(f,\tau)$, we define the pseudo-index of $C$, denoted $|ind|(f,\tau;C)$,  by:

\begin{small}

$$ \left\{
\begin{array}{cl}
\displaystyle\sum_{x_i \in C} ind(f,f\circ\tau;x_i)\ mod\ 2 & \mbox{ (if } C \mbox{ is single) and }\\
 & \mbox{(} \tau \mbox{ reverses orientation and } n \mbox{ is even;} \\
 & \mbox{  or } \\
 &  \tau \mbox{ preserves orientation and } n \mbox{ is odd.)} \\
 & \\
\dfrac{ind(f,f\circ\tau;C)}{2} & \mbox{ (if } C \mbox{ is single) and }\\
 & \mbox{(} \tau \mbox{ preserves orientation and } n \mbox{ is even;} \\
 & \mbox{  or } \\
 &  \tau \mbox{ reverses orientation and } n \mbox{ is odd.)} \\
 & \\
|ind(f,f\circ\tau;C_1)| & \mbox{(if } C \mbox{ is double, } C=C_1\cup C_2  \mbox{) and }\\
& \mbox{(}  \tau \mbox{ reverses orientation and } n  \mbox{ is even;} \\
& \mbox{  or } \\
&  \tau \mbox{ preserves orientation and } n \mbox{ is odd.)} \\
& \\
ind(f,f\circ\tau;C_1) & \mbox{(if } C \mbox{ is double, } C=C_1\cup C_2  \mbox{) and }\\
& \mbox{(}  \tau \mbox{ preserves orientation and } n  \mbox{ is even;} \\
& \mbox{  or } \\
&  \tau \mbox{ reverses orientation and } n \mbox{ is odd.)} \\
\end{array}
\right.
$$
\end{small}
where $C_1$ and $C_2$ are disjoint usual coincidence classes of the pair $(f,f\circ\tau)$.

\end{defi}

We call a Borsuk-Ulam coincidence class $C$ essential if  $|ind|(f,\tau;C)\neq 0$ and we define $NBU(f,\tau)$, the Nielsen-Borsuk-Ulam number of the pair $(f,\tau)$, as the number of essential Borsuk-Ulam coincidences classes. The definitions above are exactly what we need in order to prove:

\begin{prop}\cite[2.7]{daniel2} If $f'$ is a map homotopic to $f$ then $f'$ has at least $NBU(f,\tau)$ pairs of Borsuk-Ulam coincidence points.
\end{prop}

\section{Nielsen-Borsuk-Ulam number in $\T^1$}\label{capt1}

Our objective here is to use the classification of free involutions and compute, using good representatives, the Nielsen-Borsuk-Ulam number for each homotopy class of maps. 

We know that the $1$-torus, $\T^1$, is homeomorphic to $S^1$, with the only free involution in $S^1$ being the antipodal:
\[\tau(x)=x+\frac{1}{2}.\]

Let $f:S^1 \to S^1$ be a map, up to homotopy we can consider $f(x)=ax$ with $a \in \Z$. If $a$ is an odd number, we have $\coin(f,f\circ \tau)=\emptyset$, implying that the Nielsen-Borsuk-Ulam number is zero. Now, if $a$ is an even number then $\coin(f,f\circ \tau)=S^1$. 

We will use a small perturbation of $f$, in order to do so, we define $\epsilon:\T^1 \to \T^1$ given by $\epsilon(t)=\frac{1}{3}\sen(2\pi t)$ and  $f'(x)=ax+\epsilon(x)$. Then  $\coin(f',f'\circ \tau)=\{0,\frac{1}{2}\}$ and $\bucoin(f',\tau)=\{(0,\frac{1}{2})\}$, both coincidence points in the same coincidence class, so we have a unique single Borsuk-Ulam coincidence class with pseudo-index $|\ind|(f',f'\circ \tau;C)=1$. Therefore, $\NBU(f,\tau)=1$. By doing this, we have just proved the following result.

\begin{teo}\label{NBUtoroT1}
Let $f:S^1 \to S^1$ be a map, with $f_{\#}=\begin{pmatrix}a\end{pmatrix}$ and $a \in \Z$. Then 
\[NBU(f,\tau)=\left\{
\begin{array}{cl}
1 & \text{if a is even};\\
0 & \text{if a is odd},
\end{array}
\right.\]
where $f_{\#}$ is the homomorphism induced from $f$ in the fundamental group. Moreover, for each map $f:S^1 \to S^1$ there exists $f'$ homotopic to $f$ that realizes the $\NBU(f,\tau)$, i.e., $S^1$ is a Wecken type space in Nielsen-Borsuk-Ulam theory.
\end{teo}

\section{Nielsen-Borsuk-Ulam number in $\T^2$}\label{capt2}

We must remember that two free involutions $\tau_1$, $\tau_2$ in $M$ (closed, connected and orientable $n-$manifolds) are equivalent if there exists a homeomorphism $h:M \to M$ such that $h \circ \tau_2=\tau_1 \circ h$, i.e., the diagram 
\[\xymatrix{ M \ar[r]^{\tau_2} \ar[d]_h & M \ar[d]^h \\
     M \ar[r]_{\tau_1} & M }\]
is commutative. 

In \cite[Proposition 30, Proposition 32]{daciberg} we can find a classification for free involutions on $\T^2$, i.e, the authors proved that there are two free involutions in the torus $\T^2$, up to equivalence. Consider the following free involutions: 
\begin{align*}
\tau_1: \T^2 &\to \T^2 \qquad \qquad \qquad \qquad \tau_2: \T^2 \to \T^2\\
(x,y) & \mapsto \left(x+\frac{1}{2},y \right), \qquad \qquad
(x,y)  \mapsto \left(-x,y+\frac{1}{2}\right).
\end{align*}
The free involutions $\tau_1$ and $\tau_2$ are not equivalent, because the space of orbits $\T_{\tau_1}^2:=\frac{\T^2}{x \sim \tau_1(x)}$ is homeomorphic to $\T^2$ and the space of orbits $\T_{\tau_2}^2:=\frac{\T^2}{x \sim \tau_2(x)}$ is homeomorphic to $\K^2$ (the Klein bottle). The free involution $\tau_1$ preserves orientation, while the free involution $\tau_2$ reverses orientation.

The following theorems by \cite[Theorem 2.2.3, Theorem 2.3.10]{tesevinicius} (see \cite{vinicius} for similar results), indicate which homotopy class of $[\T^2, \T^2]$ has the Borsuk-Ulam property with respect to $\tau_1$ and $\tau_2$ involutions.

\begin{teo}\label{2.2.3}
Let $\tau_1$ be the free involution above. So every class of homotopy $\beta \in [\T^2,\T^2]$ does not have the Borsuk-Ulam property with respect to $\tau_1$.
\end{teo}

\begin{teo}\label{2.3.10}
Let $\tau_2$ be the aforementioned involution and $\beta \in [\T^2,\T^2]$ a homotopy class. Then $\beta$ has the Borsuk-Ulam property with respect to $\tau_2$ if, and only if, homomorphism $\beta_{\#}:\Z \oplus \Z \to \Z \oplus \Z$ is given by
\begin{equation*}
\beta_{\#}=
\begin{pmatrix}
m_1 & m_2\\
n_1 & n_2
\end{pmatrix}
\end{equation*}
where $(m_1,n_1) \neq (0,0)$, $m_2$ and $n_2$ are even.
\end{teo}

Those results will help us with the following:

\begin{teo}\label{NBUtoroT2}
Let $f:\T^2 \to \T^2$ be a map. Then,
\[\NBU(f,\tau_1)=0\]
and
\[\NBU(f,\tau_2)=
\left\{
\begin{array}{cl}
2 & \text{if } f_{\#}= \begin{pmatrix} p & 2k\\ q & 2l \end{pmatrix};\\
0 & \text{otherwise},
\end{array}
\right.
\]
with $p,q,k,l \in \Z$, $(p,q) \neq (0,0)$ and $f_{\#}$ being the homomorphism induced from $f$ in the fundamental group. Moreover, for each map $f:\T^2 \to \T^2$ there exists $f'$ homotopic to $f$ which realizes the $\NBU(f,\tau_i)$, i.e., the torus $\T^2$ is a Wecken type space in Nielsen-Borsuk-Ulam theory.
\end{teo}
\begin{proof}
The Theorem~\ref{2.2.3} shows that $NBU(f,\tau_1)=0$ for any $f:\T^2 \to \T^2$ and Theorem~\ref{2.3.10} describes homotopy classes such that $\NBU(f,\tau_2)=0$, easily realizing the Nielsen-Borsuk-Ulam numbers in these cases. 

Now, we just need to compute $\NBU(f,\tau_2)$ for maps such that:
\begin{equation*}
f_{\#}=
\begin{pmatrix}
p & 2k\\
q & 2l
\end{pmatrix},
\end{equation*}
with $(p,q)\neq (0,0)$. 

\begin{enumerate}
\item Suppose that $p=0$ or $q=0$. Without loss of generality, assume $q=0$. Up to homotopy we can assume that $f$ can be written as $f(x,y)=(px+2ky,2ly)$ with $p,k,l \in \Z$ and $p \neq 0$. 

The coincidence set of pair $(f,f \circ \tau_2)$ is
\[\coin(f,f \circ \tau_2)=\left\{ (x,y) \in \T^2 \mid x=\dfrac{i}{2p} \text{ with } i=0,1,\ldots,2p-1\right\},\]
because
\begin{equation*}
f(x,y)=(f \circ \tau_2)(x,y)
\Leftrightarrow 
\left\{\begin{array}{l}
px+2ky=-px+2ky+k\\
2ly=2ly+l
\end{array} \right.
\Leftrightarrow
\left\{\begin{array}{l}
2px=0\\
0=0
\end{array}. \right.
\end{equation*}

Using coverings, it is easy to see that each connected component of the set $\coin(f,f \circ \tau_2)$ is exactly one usual coincidence class of the pair $(f,f \circ \tau_2)$.

With a small perturbation $f'(x,y)=(px+2ky+\epsilon(y),2ly+\epsilon(x))$, in which $\epsilon:\T^1 \to \T^1$ is given by $\epsilon(t)=\frac{1}{n_0}\sen(2\pi t)$, $n_0 \in \Na$ and $n_0 \geqslant 3$, we have $f'$ homotopic to $f$ and 
\begin{equation*}
f'(x,y)=(f' \circ \tau_2)(x,y)
\Leftrightarrow 
\left\{\begin{array}{l}
2px+\epsilon(y)=\epsilon(y+\frac{1}{2})\\
\epsilon(x)=\epsilon(-x)
\end{array} \right.
\Leftrightarrow
\left\{\begin{array}{l}
y=0,\frac{1}{2}\\
x=0,\frac{1}{2}
\end{array}. \right.
\end{equation*}
Therefore, 
\[\coin(f',f' \circ \tau_2)=\left\{ \left(0,0\right), \left(0,\frac{1}{2}\right), \left(\frac{1}{2},0\right), \left(\frac{1}{2},\frac{1}{2}\right) \right\}\]
and the set of Borsuk-Ulam coincidences of the pair $(f',\tau_2)$ is
\[\bucoin(f',\tau_2)=\left\{
\left( \left(0,0\right), \left(0,\frac{1}{2}\right) \right), \left( \left(\frac{1}{2},0\right), \left(\frac{1}{2},\frac{1}{2}\right) \right) 
\right\}.\]

So there are two Nielsen coincidence classes $C_1=\left\{ \left(0,0\right), \left(0,\frac{1}{2}\right) \right\}$ and $C_2=\left\{ \left(\frac{1}{2},0\right), \left(\frac{1}{2},\frac{1}{2}\right) \right\}$, and two single Borsuk-Ulam coincidences classes 
$D_1=\left\{\left( \left(0,0\right), \left(0,\frac{1}{2}\right) \right)\right\}$ and $D_2=\left\{\left( \left(\frac{1}{2},0\right), \left(\frac{1}{2},\frac{1}{2}\right) \right)\right\}.$

To compute the local coincidence index we use a version that generalizes the fixed point for smooth maps presented in \cite[Definition (2.2.2)]{JezierskiMarzantowicz}. Thus, the local coincidence indices are shown in figure~\ref{figura1}.

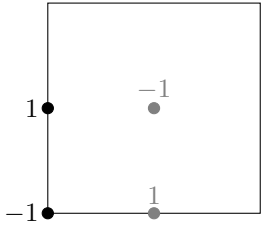
\begin{figure}[h!]\label{figura1}
\centering
\begin{tikzpicture}[scale=1.4]
\draw (0,0) rectangle (2,2);
\textcolor{black}{\draw[fill] (0,0) circle (1.5pt) node[left]{$-1$};}
\textcolor{gray}{\draw[fill] (1,0) circle (1.5pt) node[above]{$1$};}
\textcolor{black}{\draw[fill] (0,1) circle (1.5pt) node[left]{$1$};}
\textcolor{gray}{\draw[fill] (1,1) circle (1.5pt) node[above]{$-1$};}
\end{tikzpicture}
\caption{The Borsuk-Ulam coincidence classes for $f'$ with the local coincidence index of each point.}
\end{figure} 

The pseudo-index of the Borsuk-Ulam coincidence classes $D_1$ and $D_2$ are:
\[|\ind|(f,\tau_2;D_1)=\ind(f,f \circ \tau_2;(0,0)) \mod 2=1 \quad \text{and}\]
\[|\ind|(f,\tau_2;D_2)=\ind\left(f,f \circ \tau_2;\left(\frac{1}{2},0\right)\right) \mod 2=1\]  
So $D_1$ and $D_2$ are essential Borsuk-Ulam coincidence classes. Therefore,
\[\NBU(f,\tau_2)=2.\]

\item Suppose that $p \neq 0$ and $q \neq 0$. Then
\begin{align*}
f(x,y)=(f \circ \tau_2)(x,y)
&\Leftrightarrow
\left\{\begin{array}{l}
2px=0\\
2qx=0
\end{array} \right.\\
&\Leftrightarrow 
\left\{\begin{array}{lll}
x=\frac{i}{2p} & \text{with} & i=0,1,\ldots,2p-1\\
x=\frac{j}{2q} & \text{with} & j=0,1,\ldots,2q-1,
\end{array} \right. 
\end{align*}
i.e., the coincidence set of the pair $(f,f\circ \tau_2)$ is the set of points $(x,y) \in \T^2$ such that $x=\frac{i}{2p}=\frac{j}{2q}$. For example, the set $\left(0,y \right)$ and $\left(\frac{1}{2},y\right)$ are contained in the $\coin(f,f \circ \tau_2)$, so $\coin(f,f \circ \tau_2)\neq \emptyset$.

Let $a=\gcd\{p,q\}$, then $p=a\alpha$ and $q=a\beta$ for some $\alpha, \beta \in \Z$. We claim that 
$$\coin(f,f \circ \tau_2)=\left\{(x,y) \in \T^2 \mid x=\frac{k}{2a} \quad \text{with} \quad k=0,1,\ldots,2a-1 \right\}.$$

Indeed, given $\left(\frac{k}{2a},y\right) \in \T^2$, we have
\[\frac{k}{2a}=\frac{k\beta}{2a\beta}=\frac{k\beta}{2q} \quad \text{and} \quad \frac{k}{2a}=\frac{k\alpha}{2a\alpha}=\frac{k\alpha}{2p} \quad \Rightarrow \quad  \frac{k}{2a}=\frac{k\alpha}{2p}=\frac{k\beta}{2q},\]
soon $\left(\frac{k}{2a},y\right)\in \coin(f,f \circ \tau_2)$. Then, 
\[\left\{(x,y) \in \T^2 \mid x=\frac{k}{2a} \quad \text{with} \quad k=0,1,\ldots,2a-1 \right\} \subset \coin(f,f \circ \tau_2).\]
	
Now, if $(x,y) \in \coin(f,f \circ \tau_2)$ then
\[x=\frac{i}{2p}=\frac{j}{2q} \Leftrightarrow 2qi=2pj \Leftrightarrow 2a\beta i=2a\alpha j \Leftrightarrow \beta i=\alpha j.\]
Let us analyze the values of $i$ and $j$ which satisfy the above equality. For $i=0$ consider $j=0$, for $i=\alpha$ consider $j=\beta$ and successively, for $i=k\alpha$ consider $j=k\beta$, with $k=0,1,\ldots,2a-1$. Thereby, 
\[\coin(f,f \circ \tau_2) \subset \left\{(x,y) \in \T^2 \mid x=\frac{k}{2a} \quad \text{with} \quad k=0,1,\ldots,2a-1\right\},\] proving the claim.

If $(x,y) \in \coin(f,f \circ \tau_2)$, then
\begin{equation*}
f(x,y)=(f \circ \tau_2)(x,y)
\Leftrightarrow
\left\{\begin{array}{l}
x=\frac{k}{2a}\\
y=y
\end{array} \right.
\end{equation*}

Define $f':\T^2 \to \T^2$ by $f'(x,y)=(px+2ky+\epsilon(y),2ly+\epsilon(x))$. Thereby, $f'$ is homotopic to $f$ and
\begin{align}
f'(x,y)=(f' \circ \tau_2)(x,y)
&\Leftrightarrow 
\left\{\begin{array}{l}
x+\epsilon(y)=\frac{k}{2a}+\epsilon(y+\frac{1}{2})\\
\epsilon(x)=\epsilon(-x)
\end{array} \right.\\
&\Leftrightarrow \label{st2} 
\left\{\begin{array}{l}
x+\epsilon(y)=\frac{k}{2a}+\epsilon(y+\frac{1}{2})\\
x=0,\frac{1}{2}
\end{array}. \right.
\end{align}

Replacing $x=0$ in the first equation of system $\eqref{st2}$, we have $\epsilon(y)=\frac{k}{2a}+\epsilon\left(y+\frac{1}{2}\right)$. If $k=0$, the solutions in this equation are $y=0$ and $y=\frac{1}{2}$. Otherwise, for $n_0$ big enough in the definition of the map $\epsilon$, we have $\epsilon(y)\neq \frac{k}{2a}+\epsilon\left(y+\frac{1}{2}\right)$. Replacing $x=\frac{1}{2}$ in the first equation of system $\eqref{st2}$, we have $\frac{1}{2}+\epsilon(y)=\frac{k}{2a}+\epsilon\left(y+\frac{1}{2}\right)$. If $k=a$, then
$y=0$ and $y=\frac{1}{2}$ are the only solutions to this equation. Otherwise, we have $\frac{1}{2}+\epsilon(y) \neq \frac{k}{2a}+\epsilon\left(y+\frac{1}{2}\right)$ for $n_0$ big enough. Thus,
\[\coin(f',f' \circ \tau_2)=\left\{ \left(0,0\right), \left(0,\frac{1}{2}\right), \left(\frac{1}{2},0\right), \left(\frac{1}{2},\frac{1}{2}\right) \right\}\] 
and
\[\bucoin(f',\tau_2)=\left\{
\left( \left(0,0\right), \left(0,\frac{1}{2}\right) \right), \left( \left(\frac{1}{2},0\right), \left(\frac{1}{2},\frac{1}{2}\right) \right) 
\right\}.\]

Calculating the Nielsen-Borsuk-Ulam number in the same way as in the previous case, we conclude that 
\[\NBU(f,\tau_2)=2.\]
\end{enumerate}

Moreover, in both the cases $f'$ realizes the $\NBU(f,\tau_2)=2$.
\end{proof}

\section{Nielsen-Borsuk-Ulam number in $\T^3$}

In \cite{Hempel} and \cite{daciclaude} we can find classifications for free involutions on $\T^3$. Up to equivalence there are four such involutions. These are described in \cite{Hempel} by:

\begin{align*}
h_1(x,y,z)&=\left(x,y,z+\frac{1}{2}\right)\\
h_2(x,y,z)&=\left(-x,-y,z+\frac{1}{2}\right)\\
h_3(x,y,z)&=\left(x,-y,z+\frac{1}{2}\right)\\
h_4(x,y,z)&=\left(x+y,-y,z+\frac{1}{2}\right).
\end{align*}

\begin{teo}\label{NBUtoroT3}
Let $f:\T^3 \to \T^3$ be a map such that $f_{\#}:\pi_1(\T^3) \to \pi_1(\T^3)$ is represented by the matrix $f_{\#}=\begin{pmatrix}
a & b & c\\
r & s & t\\
u & v & w
\end{pmatrix}.$ Then,
\[\NBU(f,h_1)=0, \qquad \NBU(f,h_3)=0, \quad  \NBU(f,h_4)=0 \quad \text{and} \]
\[\NBU(f,h_2)=
\left\{
\begin{array}{cl}
4 & \text{if } c,t,w \text{ are even}, (a,r,u)\neq(0,0,0),\\ 
  & (b,s,v)\neq(0,0,0) \text{ and } (p,q)\neq(0,0)\\
  & \text{or}\\
  & \text{if } c,t,w \text{ are even}, (a,r,u)\neq(0,0,0),\\ 
  & (b,s,v)\neq(0,0,0), (p,q)=(0,0) \text{ and } u=0\\
0 & \text{otherwise},
\end{array}
\right.
\]
with $p=\det \begin{pmatrix}
r & s\\ u & v
\end{pmatrix}$ 
and
$q=\det \begin{pmatrix}
a & b\\ u & v
\end{pmatrix}$. Moreover, the torus $\T^3$ is a Wecken type space in the Nielsen-Borsuk-Ulam theory.
\end{teo}
\begin{obs}
The fact that the torus $\T^3$ is a Wecken space in the Nielsen-Borsuk-Ulam theory has already been demonstrated, see \cite[Theorem 3.5]{daniel2}.
\end{obs}

\begin{proof}
We can suppose (up to homotopy) that $f:\T^3 \to \T^3$ is given by  \[f(x,y,z)=(ax+by+cz,rx+sy+tz,ux+vy+wz).\]

The proof of the theorem will be presented in the following subsections:
 
\subsection{$h_1(x,y,z)=\left(x,y,z+\frac{1}{2}\right)$, $h_3(x,y,z)=\left(x,-y,z+\frac{1}{2}\right)$ and $h_4(x,y,z)=\left(x+y,-y,z+\frac{1}{2}\right)$}

The free involutions $h_1$, $h_3$ and $h_4$ can be written as
\[h_{2i+ij+1}(x,y,z)=\left(x+ijy,(-1)^iy,z+\frac{1}{2}\right) \text{ where } i,j \in \{0,1\}.\]

The equation $f(x,y,z)=(f \circ h_{2i+ij+1})(x,y,z)$ corresponds to the following system in the group $\T^3$:
\begin{equation*}
\left\{\begin{array}{l}
ax+by+cz=ax+aijy+b(-1)^iy+cz+\frac{c}{2}\\
rx+sy+tz=rx+rijy+s(-1)^iy+tz+\frac{t}{2}\\
ux+vy+wz=ux+uijy+v(-1)^iy+wz+\frac{w}{2}
\end{array} \right.
\Leftrightarrow 
\end{equation*}
\begin{equation*}
\left\{\begin{array}{l}
(b-aij-b(-1)^i)y=\frac{c}{2}\\
(s-rij-s(-1)^i)y=\frac{t}{2}\\
(v-uij-v(-1)^i)y=\frac{w}{2}
\end{array} \right.
\Leftrightarrow 
\left\{\begin{array}{l}
A_{ij}y=\frac{c}{2}\\
B_{ij}y=\frac{t}{2}\\
C_{ij}y=\frac{w}{2}
\end{array}, \right.
\end{equation*}
where $A_{ij}=b-aij-b(-1)^i$, $B_{ij}=s-rij-s(-1)^i$ and $C_{ij}=v-uij-v(-1)^i$. This way, the $\NBU(f,h_{2i+ij+1})$ calculation will be divided into a number of cases:
\begin{enumerate}
\item For $A_{ij}=B_{ij}=C_{ij}=0$ consider $f':\T^3 \to \T^3$ defined by 
$$f'(x,y,z)=\left(ax+by+cz,rx+sy+tz+\epsilon(z),ux+vy+wz+\delta(z)\right),$$         
where $\epsilon, \delta:\T^1 \to \T^1$ are given by $\epsilon(z)=\frac{1}{n_0}\sen(2\pi z)$ and $\delta(z)=\frac{1}{n_0}\sen(\pi z),$ with $n_0 \in \Na$ conveniently chosen. Note that the map $f'$ is homotopic to $f$ and the coincidence set $\coin(f',f' \circ h_{2i+ij+1})$ is formed by the points $(x,y,z) \in \T^3$ such that
\begin{equation}\label{2sisth1}
f'(x,y,z)=(f' \circ h_{2i+ij+1})(x,y,z) 
\Leftrightarrow
\left\{\begin{array}{l}
0=\frac{c}{2}\\
\epsilon(z)=\frac{t}{2}+\epsilon(z+\frac{1}{2})\\
\delta(z)=\frac{w}{2}+\delta(z+\frac{1}{2})
\end{array}. \right.
\end{equation}
If the integer $c$ is odd then the system~\eqref{2sisth1} has no solution, so we have $\coin(f',f' \circ h_{2i+ij+1})=\emptyset$. For $c$ even, we will solve system~\eqref{2sisth1} in two cases:
\begin{enumerate}
	\item $t$ odd or $w$ odd: Assume, without loss of generality, that $t$ is odd. Then, the second equation of the system~\eqref{2sisth1} is 
	\[\epsilon(z)=\frac{1}{2}+\epsilon\left(z+\frac{1}{2}\right) \qquad \Leftrightarrow \qquad \frac{2}{n_0}\sen(2\pi z)=\frac{1}{2}.\]  
	Choosing $n_0 \geqslant 5$, we have $\frac{2}{n_0}\sen(2\pi z)<\frac{1}{2}$. Thereby, $\epsilon(z) \neq \frac{1}{2}+\epsilon\left(z+\frac{1}{2}\right)$ for all $z$, concluding that $\coin(f',f' \circ h_{2i+ij+1})=\emptyset$ in this case.
	\item $t$ even and $w$ even: Then,
	\begin{equation*}
	f'(x,y,z)=(f' \circ h_{2i+ij+1})(x,y,z)  \Leftrightarrow  
	\left\{\begin{array}{l}
	0=0\\
	\epsilon(z)=\epsilon(z+\frac{1}{2})\\
	\delta(z)=\delta(z+\frac{1}{2})
	\end{array} \right.
	\Leftrightarrow  
	\left\{\begin{array}{l}
	0=0\\
	z=0,\frac{1}{2}\\
	z=\frac{1}{4},\frac{3}{4}
	\end{array}. \right.
	\end{equation*}
	concluding that $\coin(f',f' \circ h_{2i+ij+1})=\emptyset$ in this case.
\end{enumerate}
\item Two of the integers $A_{ij}$, $B_{ij}$ and $C_{ij}$ are null. Assume, without loss of generality, that $A_{ij}=0$, $B_{ij}=0$ and $C_{ij}\neq 0$. Then,
\begin{equation}\label{2sisth4}
f(x,y,z)=(f \circ h_{2i+ij+1})(x,y,z) \Leftrightarrow 
\left\{\begin{array}{l}
0=\frac{c}{2}\\
0=\frac{t}{2}\\
C_{ij}y=\frac{w}{2}
\end{array}. \right.
\end{equation}
For $c$ odd or $t$ odd, it can  immediately be concluded that the system~\eqref{2sisth4} has no solution. If both integers are even, consider $f':\T^3 \to \T^3$ defined by $f'(x,y,z)=(ax+by+cz+\epsilon(z),rx+sy+tz+\delta(z),ux+vy+wz).$  Note that $f'$ is homotopic to $f$ and $\coin(f',f' \circ h_{2i+ij+1})=\emptyset$, because
\begin{equation}\label{3sisth4}
f'(x,y,z)=(f' \circ h_{2i+ij+1})(x,y,z) \Leftrightarrow 
\left\{\begin{array}{l}
\epsilon(z)=\epsilon(z+\frac{1}{2})\\
\delta(z)=\delta(z+\frac{1}{2})\\
C_{ij}y=\frac{w}{2}
\end{array} \right.
\Leftrightarrow 
\left\{\begin{array}{l}
z=0,\frac{1}{2}\\
z=\frac{1}{4},\frac{3}{4}\\
C_{ij}y=\frac{w}{2}
\end{array}. \right.
\end{equation} 
\item Only one of the integers $A_{ij}$, $B_{ij}$ and $C_{ij}$ is null. Without loss of generality, suppose $A_{ij}=0$, $B_{ij}\neq 0$ and $C_{ij}\neq 0$. So,
\begin{equation*}
f(x,y,z)=(f \circ h_{2i+ij+1})(x,y,z) \Leftrightarrow 
\left\{\begin{array}{l}
0=\frac{c}{2}\\
B_{ij}y=\frac{t}{2}\\
C_{ij}y=\frac{w}{2}
\end{array}. \right.
\end{equation*}
Obviously, if $c$ is odd then $\coin(f,f\circ h_{2i+ij+1})=\emptyset$. Now, for $c$ even, consider $f':\T^3 \to \T^3$ defined by $f'(x,y,z)=(ax+by+cz+\delta(z),rx+sy+tz+\epsilon(z),ux+vy+wz)$. Then, 
\begin{equation}\label{4sisth4}
f'(x,y,z)=(f' \circ h_{2i+ij+1})(x,y,z) \Leftrightarrow 
\left\{\begin{array}{l}
\delta(z)=\delta(z+\frac{1}{2})\\
B_{ij}y+\epsilon(z)=\frac{t}{2}+\epsilon(z+\frac{1}{2})\\
C_{ij}y=\frac{w}{2}
\end{array} \right.
\end{equation}
Let $(x,y,z) \in \T^3$ be such that $y$ is the solution to the equation $C_{ij}y=\frac{w}{2}$. Like $C_{ij} \neq 0$, there exists a finite amount of $y$ satisfying this equation. If there exists $y$ such that $Cy=\frac{w}{2}$ and $By=\frac{t}{2}$, then the system~\eqref{4sisth4} coincides with the system~\eqref{3sisth4}. Therefore, such values of $y$ are not solutions to the system~$\eqref{4sisth4}$. For the values of $y$ such that $C_{ij}y=\frac{w}{2}$ and $B_{ij}y \neq \frac{t}{2}$, we have $By+\epsilon(z) \neq \frac{t}{2}+\epsilon(z+\frac{1}{2})$ for $n_0$ big enough. This demonstrates that all values of $y$ which satisfy the equation $Cy=\frac{w}{2}$ are not solutions to the system~\eqref{4sisth4}. Therefore, $\coin(f',f'\circ h_{2i+ij+1})=\emptyset$. 
\item If none of the integers $A_{ij}$, $B_{ij}$ and $C_{ij}$ are null. Consider $f':\T^3 \to \T^3$ defined by 
$$f'(x,y,z)=\left(ax+by+cz,rx+sy+tz+\epsilon(z),ux+vy+wz+\delta(z)\right).$$ Note that the map $f'$ is homotopic to $f$. Now,
\begin{equation}\label{novoestrela}
f'(x,y,z)=(f' \circ h_{2i+ij+1})(x,y,z)  \Leftrightarrow  
\left\{\begin{array}{l}
A_{ij}y=\frac{c}{2}\\
B_{ij}y+\epsilon(z)=\frac{t}{2}+\epsilon(z+\frac{1}{2})\\
C_{ij}y+\delta(z)=\frac{w}{2}+\delta(z+\frac{1}{2})
\end{array}. \right.
\end{equation}
Suppose that $A_{ij}y=\frac{c}{2}$ (otherwise $\coin(f',f' \circ h_{2i+ij+1})=\emptyset$). The $\coin(f',f' \circ h_{2i+ij+1})$ calculation will be divided into two cases: 
\begin{enumerate}
	\item $B_{ij}y=\frac{t}{2}$ and $C_{ij}y=\frac{w}{2}$;\\
	Thus
	\begin{equation*}
	f'(x,y,z)=(f' \circ h_{2i+ij+1})(x,y,z)  \Leftrightarrow  
	\left\{\begin{array}{l}
	0=0\\
	\epsilon(z)=\epsilon(z+\frac{1}{2})\\
	\delta(z)=\delta(z+\frac{1}{2})
	\end{array} \right.
	\Leftrightarrow  
	\left\{\begin{array}{l}
	0=0\\
	z=0,\frac{1}{2}\\
	z=\frac{1}{4},\frac{3}{4}
	\end{array}, \right.
	\end{equation*}
	concluding that the system~\eqref{novoestrela} has no solution.
	\item $B_{ij}y \neq \frac{t}{2}$ or $C_{ij}y \neq \frac{w}{2}$;\\
	Without loss of generality, suppose $B_{ij}y \neq \frac{t}{2}$. 
	Like $A_{ij} \neq 0$ and $B_{ij} \neq 0$, there exists a finite amount of $y$ satisfying the conditions $A_{ij}y=\frac{c}{2}$ and $B_{ij}y \neq \frac{t}{2}$. Choosing $n_0$ big enough we have $B_{ij}y+\epsilon(z) \neq \frac{t}{2}+\epsilon(z+\frac{1}{2})$, concluding that the system \eqref{novoestrela} has no solution.
\end{enumerate}
\end{enumerate} 

For all possible cases, $\coin(f',f' \circ h_{2i+ij+1})=\emptyset$. Therefore, \[\NBU(f,h_{2i+ij+1})=0.\]

\subsection{$h_2(x,y,z)=\left(-x,-y,z+\frac{1}{2}\right)$}

Consider the integer values
\begin{eqnarray}
&p=\det \begin{pmatrix}
r & s\\ u & v
\end{pmatrix}=rv-su, \qquad
q=\det \begin{pmatrix}
a & b\\ u & v
\end{pmatrix}=av-bu \nonumber \\
\text{and} & o=\det \begin{pmatrix}
a & b\\ r & s
\end{pmatrix}=as-br \nonumber
\end{eqnarray}

The proof for the free involution $h_2$ will be divided as follows: first $(a,r,u)=(0,0,0)$ or $(b,s,v)=(0,0,0)$, the second case $(a,r,u) \neq (0,0,0)$ and $(b,s,v) \neq (0,0,0)$ will be divided into four new cases: $c,t$ and $w$ are even integers; the set $\{c,t,w\}$ has an odd integer and two even numbers; the set $\{c,t,w\}$ has two odd numbers and an even integer; and lastly, $c,t$ and $w$ are odd integers. In these subcases, they will be analyzed when $p$ and $q$ are zeros or not.

First, if $(a,r,u)=(0,0,0)$ or $(b,s,v)=(0,0,0)$ then the equation $f(x,y,z)=(f \circ h_2)(x,y,z)$ corresponds to a system which is similar to the system studied in the previous section. So, applying the same technique used previously we can conclude that $\coin(f',f'\circ h_2)=\emptyset$ for some $f'$ is homotopic to $f$. Therefore, $\NBU(f,h_2)=0$.

In the second case $(a,r,u)\neq(0,0,0)$ and $(b,s,v) \neq (0,0,0)$. Without loss of generality, suppose $a\neq 0$. The computation of $\NBU(f,h_2)$ will be done in four sub-cases:

First subcase, suppose that $c$, $t$ and $w$ are even integers. Thus,
\begin{align}
f(x,y,z)=(f \circ h_2)(x,y,z) & \Leftrightarrow 
\left\{\begin{array}{l}
ax+by+cz=-ax-by+cz+\frac{c}{2}\\
rx+sy+tz=-rx-sy+tz+\frac{t}{2}\\
ux+vy+wz=-ux-vy+wz+\frac{w}{2}
\end{array} \right.\\
& \Leftrightarrow \label{1estrelah2}
\left\{\begin{array}{l}
2ax+2by=k_1\\
2rx+2sy=k_2\\
2ux+2vy=k_3
\end{array}, \right.
\end{align} 
with $k_1$, $k_2$, $k_3 \in \Z$. Note that $(0,0,z)$, $\left(0,\frac{1}{2},z\right)$, $\left(\frac{1}{2},0,z\right)$ and $\left(\frac{1}{2},\frac{1}{2},z\right) \in \T^3$ are solutions of the system~\eqref{1estrelah2}, for all $z$. Thereby, $\coin(f,f\circ h_2)\neq \emptyset$. 

To resolve the system~$\eqref{1estrelah2}$ in $\T^3$, i.e., find the points of coincidences between the maps $f$ and $f\circ h_2$, a different technique will be used. This technique consists of finding possible solutions to the system~$\eqref{1estrelah2}$ in $\R^3$ and projecting said solutions in $\T^3$, through natural projection $\pi:\R^3 \to \T^3$ given by $\pi(x,y,z)=[(x,y,z)]_{\T^3}$, getting the points of $\coin(f,f\circ h_2)$. In other words, the objective is to find points of coincidence between pairs of coverings  $(\widetilde{f},\widetilde{f\circ h_2})$ and project them by $\pi$ in $\T^3$, in order to find what the system~$\eqref{1estrelah2}$ solutions in $\T^3$ are. 

Note that changing the covering class of the pair $(\widetilde{f},\widetilde{f\circ h_2})$, the values $k_1$, $k_2$ and $k_3$ may vary, and so we can suppose all such numbers to be non-nulls.

Studying the system~$\eqref{1estrelah2}$ in $\R^3$, one obtains from the first and third equations that the triples $(x,y,z) \in \R^3$ with
\begin{equation*}
x=\frac{qk_1-abk_3+buk_1}{2qa} \qquad \text{and} \qquad y=\frac{ak_3-uk_1}{2q}
\end{equation*}
are possible solutions, for $q\neq 0$. Projecting these possible solutions in $\T^3$, like $(qk_1-abk_3+buk_1) \in \Z$ and $(ak_3-uk_1) \in \Z$, it follows that
\begin{equation*}
\coin(f,f\circ h_2)\subseteq S=\{(x,y,z) \in \T^3\ |\  2qax=0 \quad \text{and} \quad 2qy=0\},
\end{equation*}
since any coincidence point between the maps $f$ and $f \circ h_2$ in $\T^3$ is the image of a coincidence point of some pair of coverings $(\widetilde{f},\widetilde{f\circ h_2})$.

Consider the triples $(x,y,z) \in S \cap \coin(f,f\circ h_2)$, then
\begin{equation*}
f(x,y,z)=(f \circ h_2)(x,y,z) \Leftrightarrow 
\left\{\begin{array}{l}
x=\frac{i}{2qa}\\
y=\frac{j}{2q}\\
0=0
\end{array}, \right.
\end{equation*}
where $i \in \{0,1,2,\ldots,2qa-1\}$ and $j \in \{0,1,\ldots,2q-1\}$.

Let $f':\T^3 \to \T^3$  be defined by $f'(x,y,z)=(ax+by+cz+\epsilon(y),rx+sy+tz+\epsilon(z),ux+vy+wz+\epsilon(x))$, where $\epsilon:\T^1 \to \T^1$ is given by $\epsilon(z)=\frac{1}{n_0}\sen(2\pi z)$ with $n_0 \in \Na$ conveniently chosen. Then,
\begin{align*}
f'(x,y,z)=(f' \circ h_2)(x,y,z)& \Leftrightarrow 
\left\{\begin{array}{l}
x+\epsilon(y)=\frac{i}{2qa}+\epsilon(-y)\\
y+\epsilon(z)=\frac{j}{2q}+\epsilon(z+\frac{1}{2})\\
\epsilon(x)=\epsilon(-x)
\end{array} \right.\\
 & \Leftrightarrow 
\left\{\begin{array}{l}
y=0,\frac{1}{2}\\
z=0,\frac{1}{2}\\
x=0,\frac{1}{2}
\end{array}. \right.
\end{align*}
Therewith, $\coin(f',f'\circ h_2)=\left\{\begin{array}{l}
(0,0,0),(0,0,\frac{1}{2}),(0,\frac{1}{2},0),(0,\frac{1}{2},\frac{1}{2})\\
(\frac{1}{2},0,0),(\frac{1}{2},0,\frac{1}{2}),(\frac{1}{2},\frac{1}{2},0),(\frac{1}{2},\frac{1}{2},\frac{1}{2})
\end{array}\right\}$ 
and 
\begin{equation*}
\bucoin(f',h_2)=\left\{
\begin{array}{l}
\left[\left(0,0,0\right),\left(0,0,\frac{1}{2}\right)\right], \left[\left(0,\frac{1}{2},0\right),\left(0,\frac{1}{2},\frac{1}{2}\right)\right]\\
\left[\left(\frac{1}{2},0,0\right),\left(\frac{1}{2},0,\frac{1}{2}\right)\right], \left[\left(\frac{1}{2},\frac{1}{2},0\right),\left(\frac{1}{2},\frac{1}{2},\frac{1}{2}\right)\right]
\end{array}
\right\}.
\end{equation*}
There are four single Borsuk-Ulam coincidences classes. The pseudo-index shows that all four are essential. Therefore, $\NBU(f,h_2)=4$. 

Now, if $q=0$ we will at first assume $p=0$. Thus, $su=rv$ and $bu=av$. For $u\neq 0$ we have $s=\frac{rv}{u}$ and $b=\frac{av}{u}$. Then,
\begin{align*}
f(x,y,z)=(f \circ h_2)(x,y,z) & \Leftrightarrow 
\left\{\begin{array}{l}
2ax+2\left(\frac{av}{u}\right)y=0\\
2rx+2\left(\frac{rv}{u}\right)y=0\\
2vy=-2ux
\end{array} \right.\\
& \Leftrightarrow
\left\{\begin{array}{l}
2ax-\frac{a}{u}2ux=0 \\
2rx-\frac{r}{u}2ux=0\\
2vy=-2ux
\end{array} \right.\\
& \Leftrightarrow
\left\{\begin{array}{l}
0=0\\
0=0\\
2vy=-2ux
\end{array}. \right.
\end{align*} 

Let $f':\T^3 \to \T^3$ be defined by $f'(x,y,z)=(ax+by+cz+\delta(z),rx+sy+tz+\epsilon(z),ux+vy+wz)$. 
\begin{equation*}
f'(x,y,z)=(f' \circ h_2)(x,y,z)  \Leftrightarrow 
\left\{\begin{array}{l}
\delta(z)=\delta(z+\frac{1}{2})\\
\epsilon(z)=\epsilon(z+\frac{1}{2})\\
2vy=-2ux
\end{array} \right.\\
\Leftrightarrow
\left\{\begin{array}{l}
z=\frac{1}{4},\frac{3}{4}\\
z=0,\frac{1}{2}\\
2vy=-2ux
\end{array}. \right.
\end{equation*} 
So, the system has no solution in $\T^3$, that is, $\coin(f',f'\circ h_2)=\emptyset$. Therefore, $\NBU(f,h_2)=0$.

For $u=0$ we have $rv=0$ and $av=0$. Since $(a,r,u)\neq(0,0,0)$ we have $v=0$. Thus,
\begin{equation*}
f(x,y,z)=(f \circ h_2)(x,y,z)  \Leftrightarrow 
\left\{\begin{array}{l}
2ax+2by=0\\
2rx+2sy=0\\
0=0
\end{array}. \right.
\end{equation*} 
Let $f':\T^3 \to \T^3$ be defined by $f'(x,y,z)=(ax+by+cz,rx+sy+tz,wz+\epsilon(x))$. 
\begin{equation}
f'(x,y,z)=(f' \circ h_2)(x,y,z) \Leftrightarrow 
\left\{\begin{array}{l}
2ax+2by=0\\
2rx+2sy=0\\
\epsilon(x)=\epsilon(-x)
\end{array} \right.
\Leftrightarrow \label{sistemabs} 
\left\{\begin{array}{l}
2by=0\\
2sy=0\\
x=0,\frac{1}{2}
\end{array}. \right.
\end{equation}

The case $(b,s,v)=(0,0,0)$ has already been analyzed at the beginning of this section. So $(b,s)\neq(0,0)$, since $v=0$. Without loss of generality, assume $b\neq 0$. 

Consider $(x,y,z) \in \T^3$ with $y=\frac{i}{2b}$ satisfying the first two equations of \eqref{sistemabs} with $i \in \{0,1,2,\ldots,2b-1\}$. Note that $y=0$ and $y=\frac{1}{2}$ are examples of $y=\frac{i}{2b}$ satisfying the desired equations (for $i=0$ and $i=b$, respectively). Thus,
\begin{equation*}
f'(x,y,z)=(f' \circ h_2)(x,y,z) \Leftrightarrow 
\left\{\begin{array}{l}
y=\frac{i}{2b}\\
0=0\\
x=0,\frac{1}{2}
\end{array}. \right.
\end{equation*}
Let $f'':\T^3 \to \T^3$ be given by 
$$f''(x,y,z)=(ax+by+cz+\epsilon(z),rx+sy+tz+\epsilon(y),wz+\epsilon(x)).$$ Then, $f''$ is homotopic to $f$ the
\begin{align*}
f''(x,y,z)=(f'' \circ h_2)(x,y,z) & \Leftrightarrow 
\left\{\begin{array}{l}
y+\epsilon(z)=\frac{i_\theta}{2b}+\epsilon(z+\frac{1}{2})\\
y=0,\frac{1}{2}\\
x=0,\frac{1}{2}
\end{array} \right.\\ 
 & \Leftrightarrow
\left\{\begin{array}{l}
z=0,\frac{1}{2}\\
y=0,\frac{1}{2}\\
x=0,\frac{1}{2}
\end{array} \right.
\end{align*}
Therefore, 
\begin{equation*}
\bucoin(f'',h_2)=\left\{
\begin{array}{l}
\left[\left(0,0,0\right),\left(0,0,\frac{1}{2}\right)\right], \left[\left(0,\frac{1}{2},0\right),\left(0,\frac{1}{2},\frac{1}{2}\right)\right]\\
\left[\left(\frac{1}{2},0,0\right),\left(\frac{1}{2},0,\frac{1}{2}\right)\right], \left[\left(\frac{1}{2},\frac{1}{2},0\right),\left(\frac{1}{2},\frac{1}{2},\frac{1}{2}\right)\right]
\end{array}
\right\} 
\end{equation*}
with all four Borsuk-Ulam coincidence classes essential. Therefore, \[\NBU(f,h_2)=4.\]
If $q=0$ and $p\neq 0$, we have $bu=av$ and $su \neq rv$. If $u \neq 0$ then $b=\frac{av}{u}$. That way,
\begin{equation*}
f(x,y,z)=(f \circ h_2)(x,y,z)  \Leftrightarrow 
\left\{\begin{array}{l}
2ax+2\left(\frac{av}{u}\right)y=0\\
2rx+2sy=0\\
2vy=-2ux
\end{array} \right. \Leftrightarrow 
\left\{\begin{array}{l}
0=0\\
2rx+2sy=0\\
2vy=-2ux
\end{array}. \right.
\end{equation*} 
Consider $f':\T^3 \to \T^3$ given by $$f'(x,y,z)=(ax+by+cz+\epsilon(x),rx+sy+tz,ux+vy+wz).$$ 
\begin{equation*}
f'(x,y,z)=(f' \circ h_2)(x,y,z)  \Leftrightarrow 
\left\{\begin{array}{l}
\epsilon(x)=\epsilon(-x)\\
2rx+2sy=0\\
2vy=-2ux
\end{array} \right. \Leftrightarrow 
\left\{\begin{array}{l}
x=0,\frac{1}{2}\\
2sy=0\\
2vy=0
\end{array}. \right.
\end{equation*} 

We know $(s,v) \neq (0,0)$ since $p\neq 0$. Moreover, $y=0$ and $y=\frac{1}{2}$ satisfy the above system. Using the same technique applied to the system \eqref{sistemabs}, for $f'':\T^3 \to \T^3$ defined by $f''(x,y,z)=(ax+by+cz+\epsilon(x),rx+sy+tz+\epsilon(z),ux+vy+wz+\epsilon(y))$, we conclude that $\NBU(f,h_2)=4$.

If $u=0$ we have $av=0$. From the beginning $a\neq 0$, so $v=0$, which contradicts the fact that $p\neq 0$. Therefore, $u \neq 0$.

In the second subcase, the set $\{c,t,w\}$ has an odd integer and two even numbers. Without loss of generality, suppose $c$ odd. So,
\begin{equation}\label{casoumimpar}
f(x,y,z)=(f \circ h_2)(x,y,z) \Leftrightarrow 
\left\{\begin{array}{l}
2ax+2by=k_1+\frac{1}{2}\\
2rx+2sy=k_2\\
2ux+2vy=k_3
\end{array}, \right.
\end{equation} 
with $k_1$, $k_2$, $k_3 \in \Z$. We will find possible system solutions \eqref{casoumimpar} in $\R^3$ and then project them in $\T^3$. Studying the system \eqref{casoumimpar}, from the first and third equations,  we have that $(x,y,z) \in \R^3$ with
\begin{equation*}
x=\frac{2qk_1+q-2abk_3+2buk_1+bu}{4qa} \qquad \text{and} \qquad y=\frac{2ak_3-2uk_1-u}{4q}
\end{equation*}
being possible solutions, for $q\neq 0$. 

Projecting these possible solutions in $\T^3$, we have
\begin{equation*}
\coin(f,f\circ h_2)\subseteq S=\{(x,y,z) \in \T^3\ |\  4qax=0 \quad \text{and} \quad 4qy=0\},
\end{equation*}
since $(2qk_1+q-2abk_3+2buk_1+bu) \in \Z$ and $(2ak_3-2uk_1-u) \in \Z$. 

Consider $(x,y,z) \in S \cap \coin(f,f\circ h_2)$, then 
\begin{equation*}
f(x,y,z)=(f \circ h_2)(x,y,z) \Leftrightarrow 
\left\{\begin{array}{l}
x=\frac{i}{4qa}\\
y=\frac{j}{4q}\\
0=0
\end{array}, \right.
\end{equation*}
where $i \in \{0,1,2,\ldots,4qa-1\}$ and $j \in \{0,1,\ldots,4q-1\}$.

Define $f':\T^3 \to \T^3$ by $f'(x,y,z)=(ax+by+cz+,rx+sy+tz,ux+vy+wz+\epsilon(x))$. 
\begin{equation}\label{1imparf'}
f'(x,y,z)=(f' \circ h_2)(x,y,z) \Leftrightarrow 
\left\{\begin{array}{l}
x=\frac{i}{4qa}\\
y=\frac{j}{4q}\\
x=0,\frac{1}{2}
\end{array}. \right.
\end{equation}
If $(0,y,z)$ and $\left(\frac{1}{2},y,z\right)$ are not in $\coin(f,f\circ h_2)$, then the system \eqref{1imparf'} has no solution. So, $\coin(f',f'\circ h_2)=\emptyset$. Therefore, $\NBU(f,h_2)=0$. If $(0,y,z)$ and $\left(\frac{1}{2},y,z\right)$ are in $\coin(f,f\circ h_2)$, for some $y$ and $z$, so considering $f'':\T^3 \to \T^3$ given by $f''(x,y,z)=(ax+by+cz+\epsilon(y),rx+sy+tz,ux+vy+wz+\epsilon(x))$, then we have that $f''$ is homotopic to $f$ and 
\begin{equation*}
f''(x,y,z)=(f'' \circ h_2)(x,y,z) \Leftrightarrow 
\left\{\begin{array}{l}
\epsilon(y)=\epsilon(-y)\\
y=\frac{j}{4q}\\
x=0,\frac{1}{2}
\end{array} \right. \Leftrightarrow 
\left\{\begin{array}{l}
y=0,\frac{1}{2}\\
y=\frac{j}{4q}\\
x=0,\frac{1}{2}
\end{array}. \right.
\end{equation*}
However $(0,0,z)$, $\left(\frac{1}{2},0,z\right)$, $\left(0,\frac{1}{2},z\right)$, $\left(\frac{1}{2},\frac{1}{2},z\right) \in \T^3$ are not in $\coin(f,f\circ h_2)$, since they do not satisfy the first equation of the system \eqref{casoumimpar}. Therewith, $\frac{j}{4q}\neq 0$ and $\frac{j}{4q}\neq \frac{1}{2}$. Therefore $\coin(f'',f'' \circ h_2)=\emptyset$ and $\NBU(f,h_2)=0$.

For $q=0$ suppose, firstly, that $p=0$. Thus, $su=rv$ and $bu=av$. Considering $u\neq 0$ we have $s=\frac{rv}{u}$ and $b=\frac{av}{u}$. Therewith,
\begin{align*}
f(x,y,z)=(f \circ h_2)(x,y,z) & \Leftrightarrow 
\left\{\begin{array}{l}
2ax+2\left(\frac{av}{u}\right)y=\frac{1}{2}\\
2rx+2\left(\frac{rv}{u}\right)y=0\\
2vy=-2ux
\end{array} \right.\\
& \Leftrightarrow
\left\{\begin{array}{l}
2ax-\frac{a}{u}2ux=\frac{1}{2} \\
2rx-\frac{r}{u}2ux=0\\
2vy=-2ux
\end{array} \right.\\
& \Leftrightarrow
\left\{\begin{array}{l}
0=\frac{1}{2}\\
0=0\\
2vy=-2ux
\end{array}. \right.
\end{align*}
So, $\coin(f,f\circ h_2)=\emptyset$. Therefore, $\NBU(f,h_2)=0$. 

Considering $u=0$ we have $rv=0$ and $av=0$. Since $(a,r,u)\neq(0,0,0)$, it follows that $v=0$. Thus,
\begin{equation*}
f(x,y,z)=(f \circ h_2)(x,y,z)  \Leftrightarrow 
\left\{\begin{array}{l}
2ax+2by=\frac{1}{2}\\
2rx+2sy=0\\
0=0
\end{array}. \right.
\end{equation*} 
Let $f':\T^3 \to \T^3$ be defined by $f'(x,y,z)=(ax+by+cz,rx+sy+tz,wz+\epsilon(x))$. 
\begin{equation}
f'(x,y,z)=(f' \circ h_2)(x,y,z) \Leftrightarrow 
\left\{\begin{array}{l}
2ax+2by=\frac{1}{2}\\
2rx+2sy=0\\
\epsilon(x)=\epsilon(-x)
\end{array} \right.
\Leftrightarrow \label{sistemabsimpar} 
\left\{\begin{array}{l}
2by=\frac{1}{2}\\
2sy=0\\
x=0,\frac{1}{2}
\end{array}. \right.
\end{equation}
If $b=0$ then $\NBU(f,h_2)=0$. For $b\neq 0$ consider $y=\frac{i}{4b}$ satisfying the first two equations of \eqref{sistemabsimpar}, with $i \in \{1,3,5,\ldots,4b-1\}$, note that $y=0$ and $y=\frac{1}{2}$ do not satisfy such condition. Thus,
\begin{equation*}
f'(x,y,z)=(f' \circ h_2)(x,y,z) \Leftrightarrow 
\left\{\begin{array}{l}
y=\frac{i}{4b}\\
0=0\\
x=0,\frac{1}{2}
\end{array}. \right.
\end{equation*}
Let $f'':\T^3 \to \T^3$ be given by $f''(x,y,z)=(ax+by+cz,rx+sy+tz+\epsilon(y),wz+\epsilon(x))$. 
\begin{equation*}
f''(x,y,z)=(f'' \circ h_2)(x,y,z) \Leftrightarrow 
\left\{\begin{array}{l}
y=\frac{i}{4b}\\
\epsilon(y)=\epsilon(-y)\\
x=0,\frac{1}{2}
\end{array} \right. \Leftrightarrow
\left\{\begin{array}{l}
y=\frac{i}{4b}\\
y=0,\frac{1}{2}\\
x=0,\frac{1}{2}
\end{array}. \right.
\end{equation*}
Therefore, $\coin(f'',f'' \circ h_2)=\emptyset$ and $\NBU(f,h_2)=0$.

Now, if $q=0$ and $p \neq 0$ then $bu=av$ and $su \neq rv$. If $u \neq 0$ then $b=\frac{av}{u}$. That way,
\begin{equation*}
f(x,y,z)=(f \circ h_2)(x,y,z)  \Leftrightarrow 
\left\{\begin{array}{l}
2ax+2\left(\frac{av}{u}\right)y=\frac{1}{2}\\
2rx+2sy=0\\
2vy=-2ux
\end{array} \right. \Leftrightarrow 
\left\{\begin{array}{l}
0=\frac{1}{2}\\
2rx+2sy=0\\
2vy=-2ux
\end{array}. \right.
\end{equation*}
Therefore, $\coin(f,f\circ h_2)=\emptyset$ and $\NBU(f,h_2)=0$. 

If $u=0$ we have $av=0$. From the beginning $a\neq 0$, so $v=0$, contradicting the fact that $p\neq 0$. Therefore, $u \neq 0$.

The third subcase, in which the set $\{c,t,w\}$ has two odd numbers and an even number, is similar to the second case, and the Nielsen-Borsuk-Ulam number is also zero.

In the fourth and last subcase, $c,t,w$ are odd. For $q\neq 0$ we conclude, similarly to the previous cases, that $\NBU(f,h_2)=0$. 

If $q=0$, $p=0$ and $u \neq 0$, we have
\begin{align*}
f(x,y,z)=(f \circ h_2)(x,y,z) & \Leftrightarrow 
\left\{\begin{array}{l}
2ax+2\left(\frac{av}{u}\right)y=\frac{1}{2}\\
2rx+2\left(\frac{rv}{u}\right)y=\frac{1}{2}\\
2vy=\frac{1}{2}-2ux
\end{array} \right.\\
& \Leftrightarrow
\left\{\begin{array}{l}
2ax+\frac{a}{2u}-\frac{a}{u}2uy=\frac{1}{2}\\
2rx+\frac{r}{2u}-\frac{r}{u}2uy=\frac{1}{2}\\
2vy=\frac{1}{2}-2ux
\end{array} \right.\\
& \Leftrightarrow
\left\{\begin{array}{l}
\frac{a}{2u}=\frac{1}{2}\\
\frac{r}{2u}=\frac{1}{2}\\
2vy=\frac{1}{2}-2ux
\end{array}. \right.
\end{align*} 

If $\frac{a}{2u}\neq \frac{1}{2}$ or $\frac{r}{2u}\neq \frac{1}{2}$, it follows that $\NBU(f,h_2)=0$, because the system above has no solution in $\T^3$. Otherwise, if $\frac{a}{2u}=\frac{1}{2}$ and $\frac{r}{2u}=\frac{1}{2}$ then $\NBU(f,h_2)=0$, just consider $f':\T^3 \to \T^3$ defined by $f'(x,y,z)=(ax+by+cz+\delta(z),rx+sy+tz+\epsilon(z),ux+vy+wz)$.

If $q=0$, $p=0$ and $u=0$, we have $v=0$. Then,
\begin{equation*}
f(x,y,z)=(f \circ h_2)(x,y,z)  \Leftrightarrow 
\left\{\begin{array}{l}
2ax+2by=\frac{1}{2}\\
2rx+2sy=\frac{1}{2}\\
0=0
\end{array}. \right.
\end{equation*} 
Define $f':\T^3 \to \T^3$ by $f'(x,y,z)=(ax+by+cz,rx+sy+tz,wz+\epsilon(x))$. So, $f'$ is homotopic to $f$ and 
\begin{equation*}
f'(x,y,z)=(f' \circ h_2)(x,y,z)  
\Leftrightarrow 
\left\{\begin{array}{l}
2ax+2by=\frac{1}{2}\\
2rx+2sy=\frac{1}{2}\\
x=0,\frac{1}{2}
\end{array} \right.
\Leftrightarrow 
\left\{\begin{array}{l}
2by=\frac{1}{2}\\
2sy=\frac{1}{2}\\
x=0,\frac{1}{2}
\end{array}. \right.
\end{equation*} 
We know $(b,s)\neq(0,0)$, because $(b,s,v)\neq(0,0,0)$. Clearly the $\NBU(f,h_2)=0$ if $b=0$ or $s=0$. For $b\neq 0$ and $s\neq 0$, we can use the same technique applied to the system~$\eqref{sistemabsimpar}$, and conclude that $\NBU(f,h_2)=0.$

Lastly, for $q=0$ and $p\neq 0$, we have $u \neq 0$, $b=\frac{av}{u}$ and $su\neq rv$. Thus,
\begin{equation*}
f(x,y,z)=(f \circ h_2)(x,y,z) \Leftrightarrow 
\left\{\begin{array}{l}
\frac{a}{2u}=\frac{1}{2}\\
2rx+2sy=\frac{1}{2}\\
2vy=\frac{1}{2}-2ux
\end{array}. \right.
\end{equation*}
If $\frac{a}{2u} \neq \frac{1}{2}$, it follows that $\coin(f,f\circ h_2)=\emptyset$. Otherwise, let $f':\T^3 \to \T^3$ be given by $f'(x,y,z)=(ax+by+cz+\epsilon(x),rx+sy+tz,ux+vy+wz)$. Then,
\begin{equation}\label{sistimpar}
f'(x,y,z)=(f' \circ h_2)(x,y,z) \Leftrightarrow 
\left\{\begin{array}{l}
x=0,\frac{1}{2}\\
2sy=\frac{1}{2}\\
2vy=\frac{1}{2}
\end{array}. \right.
\end{equation}
Note that $(s,v)\neq(0,0)$ (because $p\neq 0$) and  $(x,y,z) \in \T^3$ with $y=0$ or $y=\frac{1}{2}$ are not solutions to the system above. If $s=0$ or $v=0$ then $\coin(f',f' \circ h_2)=\emptyset$. If $s\neq 0$ and $v\neq 0$, consider  $(x,y,z) \in \T^3$ such that $y=\frac{i}{4s}$, satisfies the last two equations of the system \eqref{sistimpar}, where $i \in \{1,3,\ldots,4s-1\}$. So,
\begin{equation*}
f'(x,y,z)=(f' \circ h_2)(x,y,z) \Leftrightarrow 
\left\{\begin{array}{l}
x=0,\frac{1}{2}\\
y=\frac{i}{4s}\\
0=0
\end{array}. \right.
\end{equation*}
For $f'':\T^3 \to \T^3$ defined by $f''(x,y,z)=(ax+by+cz+\epsilon(x),rx+sy+tz,ux+vy+wz+\epsilon(y))$.
\begin{equation}\label{sistimpar2}
f''(x,y,z)=(f'' \circ h_2)(x,y,z) \Leftrightarrow 
\left\{\begin{array}{l}
x=0,\frac{1}{2}\\
y=\frac{i}{4s}\\
y=0,\frac{1}{2}
\end{array}. \right.
\end{equation}
Since $y=0$ and $y=\frac{1}{2}$ do not satisfy the second equation of \eqref{sistimpar2}, it follows that $\coin(f'',f''\circ h_2)=\emptyset$, and $\NBU(f,h_2)=0$.
\end{proof}

\section{Nielsen-Borsuk-Ulam number in $\T^n$, $n>3$}

For the $n-$torus, with $n>3$, there is no classification of free involutions in the literature. That way, the study of the Nielsen-Borsuk-Ulam number in this particular space cannot be made in general. What we can do is consider a free involution $\tau$ and calculate the Borsuk-Ulam number $\NBU(f,\tau)$ for any map $f:\T^n \to \T^n$. 

Consider the following free involutions in $\T^n$: 
\begin{align*}
\tau_1(x_1,x_2,\ldots,x_{n-1},x_n)&=\left(x_1,x_2,\ldots,x_{n-1},x_n+\frac{1}{2}\right)\\
\tau_3(x_1,x_2,\ldots,x_{n-1},x_n)&=\left(x_1,x_2,\ldots,x_{n-2},-x_{n-1},x_n+\frac{1}{2}\right)\\
\tau_4(x_1,x_2,\ldots,x_{n-1},x_n)&=\left(x_1+x_2,-x_2,x_3,\ldots,x_{n-1},x_n+\frac{1}{2}\right).
\end{align*}
Observe that the mentioned involutions are generalizations to the $n-$torus of the involutions $h_1$, $h_3$ and $h_4$ of $3-$torus. Applying the same method used previously in $\T^3$ to these involutions, we can demonstrate that the Nielsen-Borsuk-Ulam number is zero for those involutions, i.e., for any map $f:\T^n \to \T^n$ we have
\[\NBU(f,\tau_1)=0, \qquad \NBU(f,\tau_3)=0 \quad \text{and} \quad \NBU(f,\tau_4)=0.\]

For the free involution 
$$\tau_2(x_1,x_2,\ldots,x_{n-1},x_n)=\left(-x_1,-x_2,\ldots,-x_{n-1},x_n+\frac{1}{2}\right)$$
 in $\T^n$, we have $\NBU(f,\tau_2)\neq 0$ for some map $f:\T^n \to \T^n$. Indeed, let $g:\T^n \to \T^n$ be a map such that
\begin{equation*}
g_{\#}=\begin{pmatrix}
1 & 0 & 0 & \cdots & 0 & 2b_1\\
0 & 1 & 0 & \cdots & 0 & 2b_2\\  
0 & 0 & 1 & \cdots & 0 & 2b_3\\
\vdots & \vdots & \vdots & \ddots & \vdots &  \vdots\\
0 & 0 & 0 & \cdots & 1 & 2b_{n-1}\\
0 & 0 & 0 & \cdots & 0 & 2b_n
\end{pmatrix}
\end{equation*}
where $b_i \in \Z$, i.e., \[g(x_1,\ldots,x_{n-1},x_n)=\left(x_1,\ldots,x_{n-1},2b_1x_1+2b_2x_2+\ldots+2b_nx_n\right).\] 
Let $g':\T^n \to \T^n$ be defined by 
\[g'(x_1,\ldots,x_{n-1},x_n)=\left(x_1,\ldots,x_{n-1},2b_1x_1+2b_2x_2+\ldots+2b_nx_n+\epsilon(x_n)\right),\]
where $\epsilon:\T^1 \to \T^1$ is given by $\epsilon(x)=\frac{1}{n_0}\sen(2\pi x)$, with $n_0 \in \Na$ conveniently chosen.
Note that $g'$ is homotopic to $g$ and 
\[
g'(x_1,\ldots,x_n)=(g' \circ \tau_2)(x_1,\ldots,x_n)  \Leftrightarrow
\left\{ \begin{array}{c}
x_1=0,\frac{1}{2}\\
x_2=0,\frac{1}{2}\\
\vdots\\
x_n=0,\frac{1}{2}
\end{array}. \right.
\]
Then, we have that the cardinality of the coincidence set of pair $(g',g'\circ \tau_2)$ is equal to $2^n$, $\#\coin(g',g'\circ \tau_2)=2^n$, and the cardinality of the Borsuk-Ulam coincidence set of pair $(g',\tau_2)$ is  $2^{n-1}$, $\#\bucoin(g',\tau_2)=2^{n-1}$. Therefore, there exists $2^{n-1}$ essential Borsuk-Ulam coincidence classes. Thus, we can conclude that $\NBU(g,\tau_2)=2^{n-1}$.

The results obtained here for the Nielsen-Borsuk-Ulam number in low dimension $n-$torus, $n=1,2,3$, and the example of the map $g$ in $\T^n$ above, induces the formulation of the following conjecture:

\begin{conjectura}
Let $f:\T^n \to \T^n$ be a map and $\tau$ a free involution in $\T^n$. Then
\[\NBU(f,\tau)=
\left\{
\begin{array}{cl}
2^{n-1} & \text{or }\\
0. &
\end{array}
\right.
\]
\end{conjectura}


\begin{thebibliography}{10}

\bibitem{daciclaude}
A. Bauval, D. L. Gon\c{c}alves, and C. Hayat,
\textit{The {B}orsuk-{U}lam theorem for the {S}eifert manifolds having flat
  geometry.} {https://arxiv.org/pdf/1807.00159.pdf}, 2018.

\bibitem{daniel}
F. S. Cotrim and D. Vendr\'{u}scolo.
\textit{ Nielsen coincidence theory applied to {B}orsuk-{U}lam geometric
  problems.}  Topology Appl., \textbf{159(18)} (2012), 3738--3745.

\bibitem{daniel2}
F. S. Cotrim and D. Vendr\'{u}scolo.
\textit{The {N}ielsen {B}orsuk-{U}lam number.}
 Bull. Belg. Math. Soc. Simon Stevin,\textbf{24(4)} (2017), 613--619.

\bibitem{dacibergsurfaces}
D. L. Gon\c{c}alves.
\textit{The {B}orsuk-{U}lam theorem for surfaces.}
Quaest. Math., \textbf{29(1)} (2006), 117--123.

\bibitem{daciberg}
D. L. Gon\c{c}alves and J. Guaschi.
\textit{The {B}orsuk-{U}lam theorem for maps into a surface.}
Topology Appl., \textbf{157(10-11)} (2010), 1742--1759.

\bibitem{vinicius}
D. L. Gon\c{c}alves, J. Guaschi, and V. C. Laass.
\textit{The {B}orsuk-{U}lam property for homotopy classes of self-maps of
  surfaces of {E}uler characteristic zero.}
J. Fixed Point Theory Appl., \textbf{21(2)} (2019) Art. 65.

\bibitem{Hempel}
J. Hempel.
\textit{Free cyclic actions on {$S^{1}\times S^{1}\times S^{1}$}.}
Proc. Amer. Math. Soc., \textbf{48} (1975), 221--227.

\bibitem{JezierskiMarzantowicz}
J. Jezierski and W. Marzantowicz.
\textit{ Homotopy methods in topological fixed and periodic points
  theory}, volume~3 of {\em Topological Fixed Point Theory and Its
  Applications}. Springer, Dordrecht, 2006.

\bibitem{tesevinicius}
V. C. Laass.
\textit{A propriedade de Borsuk-Ulam para fun\c c\~oes entre
  superf\'icies}. PhD thesis, Instituto de Matem\'atica e Estat\'istica da Universidade de S\~ao Paulo, 2015.

\bibitem{UsingBorsukUlam}
J. Matou\v{s}ek.
\textit{Using the {B}orsuk-{U}lam theorem}. Universitext. Springer-Verlag, Berlin, 2003. 

\end{thebibliography}
\end{document}